\documentclass[12pt, preprint]{amsart}
\usepackage[top=0.9in, bottom=1in, left=1in, right=1in]{geometry}
\usepackage{amsmath,amscd}
\usepackage{amsthm}
\usepackage{amsfonts}
\usepackage{amssymb, mathrsfs}
\usepackage{tikz}
\usepackage{tikz-cd}
\usepackage{mathtools}
\usetikzlibrary{positioning}
\usepackage{graphicx,rotating}
\usepackage{esint}
\usepackage{float}
\usepackage{hyperref}
\usepackage{enumerate}
\usepackage{blkarray}

\title[Asymptotic Schur orthogonality relations for Heisenberg group]{Asymptotic Schur orthogonality relations for Heisenberg groups over local fields}
\author{Malay Mandal}
\address{Malay Mandal
\newline
The Institute of Mathematical Sciences, X6VW+FP7, 4th Cross St, CIT Campus, Tharamani, Chennai, Tamil Nadu 600113.}
\email{malay10mandal@gmail.com}

\author{Arghya Mondal}
\address{Arghya Mondal
\newline
IISER Mohali,  Knowledge city, Sector 81, Manauli, PO, Sahibzada Ajit Singh Nagar, Punjab 140306}
\email{arghyamondal@iisermohali.ac.in}
\email{mondalarghya1990@gmail.com}
\urladdr{https://web.iisermohali.ac.in/dept/math/people/arghyamondal}
\thanks{AM is supported in part by a Department of Science and Technology Inspire Faculty Fellowship.}
\date{\today}
\newtheorem{theorem}{Theorem}[section]
\newtheorem{lemma}[theorem]{Lemma}
\newtheorem{prop}[theorem]{Proposition}

\newenvironment{definition}[1][Definition]{\begin{trivlist}
\item[\hskip \labelsep {\bfseries #1}]}{\end{trivlist}}
\newenvironment{remark}[1][Remark]{\begin{trivlist}
\item[\hskip \labelsep {\bfseries #1}]}{\end{trivlist}}

\def\r{\mathbb{R}}
\def\c{\mathbb{C}}

\def\n{\mathbb{N}}

\begin{document}
\begin{abstract} Asymptotic Schur orthogonality relations are for irreducible unitary representations of locally compact groups that need not be discrete series, where $L^2$ pairing of matrix coefficients with respect to Haar measure is replaced by a limit of that with respect to a sequence of bounded measures. We show that such relations hold for Heisenberg groups over local fields. This is achieved in the framework of c-temperedness introduced by Kazhdan and Yom Din. The related condition of convergence to braiding operator is also shown. 
    
\end{abstract}
\maketitle
\section{Introduction}

\noindent Schur orthogonality relations state that, for locally compact unimodular groups, the $L^2$ inner product of matrix coefficients of distinct discrete series representations is orthogonal, and that of matrix coefficients from the same discrete series representation has a simple description in terms of inner products of the vectors involved. For irreducible unitary representations that are not discrete series, one may try to replace the $L^2$-pairings (i.e. $(f_1,f_2)\mapsto f_1\overline{f_2}$) of matrix coefficients with respect to Haar measure $\mu$ by a limit of $L^2$-pairings with respect to a sequence $\{\mu_n\}_n$ of appropriate measures. These measures are generally taken to be absolutely continuous with respect to the Haar measure. This approach had been taken by Midorikawa \cite{mid1, mid2} and Ankabout \cite{ankabout} for tempered representations of semisimple Lie groups, where the Radon-Nikodym derivatives of their chosen measures have full support. The resulting \textit{asymptotic} Schur orthogonality relations are valid for vectors in the dense subspace of $K$-finite vectors. The next such result was proved by Boyer and Garncarek \cite{bg} for boundary representations of hyperbolic groups, where the Radon-Nikodym derivatives are supported in an annulus shaped subsets around identity, with respect to the word metric.  Their asymptotic Schur orthogonality relations are valid for all vectors.

\vspace{0.2cm}
\noindent A general framework for deriving asymptotic Schur orthogonality relations was given by Kazhdan and Yom Din in \cite{kyd} by introducing the notion of \textit{c-temperedness} of a representation. The definition of c-temperedness formalizes the following idea. Given an irreducible unitary representation $\pi$, if there exists a sequence $\{F_n\}_n$ of pre-compact subsets  of $G$ for which the average growth rates of the norm-square of all non-zero matrix coefficients of $\pi$ are roughly the same and the appropriately scaled averages of these functions over these subsets are asymptotically $G$-invariant, then the proof of classical Schur Orthogonality can be adapted to get asymptotic Schur orthogonality relations for $\pi$. In particular, the Radon Nikodym derivatives are of the form $(1/a_n)1_{F_n}$, where $a_n$ are the scaling factors. See \S\ref{ctemp} for details. They conjectured that asymptotic Schur orthogonality relations with respect to such measures should hold for all vectors in tempered representations of semisimple algebraic groups over local fields and proved it when $G=\text{SL}_2(\r)$ or $\text{PSL}_2(\Omega)$, where $\Omega$ is a non-Archimedean field of characteristic 0 and of residual characteristic not equal to 2. They also proved it for all such groups when $\pi$ is the quasi-regular representation $L^2(
G/P)$, where $P$ a minimal parabolic
subgroup. Without these restriction, the statement of the conjecture holds only for $K$-finite vectors - the case of non-archimedean fields was shown in the same paper, while the archimedean case was completed by Aubert and La Rossa \cite{aublaros} based on that work.

\vspace{0.2cm}
\noindent A rationale for the existence of a sequence of measures, that gives asymptotic Schur orthogonality relations for a given irreducible unitary representation, was given by Bendikov, Boyer and Pittet in the Introduction of \cite{bbp}. They observed that asymptotic Schur orthogonality  will hold for a representation $(\pi,V_\pi)$ with respect to bounded Borel measures $\{\mu_n\}_n$ if they satisfy the following convergence in weak operator topology on $\mathcal{B}(V_\pi\otimes V_\pi)$:
\begin{equation}\label{convtoF}
\lim_{n\to\infty}\int_G\pi(g^{-1})\otimes\pi(g)d\mu_n=F,
\end{equation}
where $F$ is the \textit{braiding} or \textit{flip} operator given by $F(v\otimes w):=w\otimes v$. They showed that (\ref{convtoF}) holds for the quasi regular representation of semisimple algebraic groups $G$ over local fields on $L^2(G/P)$, where $P$ is a minimal parabolic subgroup, and that of automorphism groups of regular trees. 

\vspace{0.2cm}
\noindent The starting point of our investigation is the following observation. For \textit{finite dimensional} irreducible unitary representations $(\pi,V_\pi)$ of second countable locally compact amenable groups there exists a F\o lner sequence $\{F_n\}_n$ such that the measures $d\mu_n:=(\dim V_\pi /\mu(F_n))1_{F_n}d\mu$ satisfy (\ref{convtoF}). Such representations are c-tempered with respect to $\{F_n\}_n$ and asymptotic Schur orthogonality relations hold for them with respect to these measures. See Theorem \ref{findim}.
%We note that all unitary representations of a countable group are finite dimensional if and only if the group is virtually abelian. 

\vspace{0.2cm}
\noindent The above observation motivates us to consider \textit{infinite dimensional} irreducible representations of amenable groups. In this paper we study the case of Heisenberg groups over local fields and show that there exists a F\o lner sequence such that all unitary irreducible representations of such groups are c-tempered with respect to this sequence, hence asymptotic Schur orthogonality relations hold for them. Further, (\ref{convtoF}) holds for measures that are supported on the sets in the F\o lner sequence. See Theorem \ref{hei}. This  Representation theory of Heisenberg groups is well studied with motivation coming from both physics and mathematics. In particular, the Fourier-Wigner transform and related ideas (see \S\ref{secfw}) carry through the proof of Theorem \ref{hei}.

\vspace{0.2cm}
\noindent Theorem \ref{hei} shows that the family of Heisenberg groups over local fields is an answer to  \cite[Question 1.17]{kyd}: For which groups $G$ every tempered irreducible unitary $G$-representation is c-tempered with some F\o lner sequence? Theorem \ref{findim} also provides examples: second countable \textit{Moore groups}, that is, locally compact groups all whose irreducible representations are finite dimensional. See \cite[Theorem 4.C.1]{BdlH} for characterization and examples of Moore groups. 

\vspace{0.2cm}
\noindent Following is a section wise description of the paper. Section \ref{prelimctemp} gathers the results we need about c-temperedness and convergence to the braiding operator. Section \ref{prelimheisenberg} covers the material we need about unitary representation theory of Heisenberg group. We included brief descriptions of induced representations (with emphasis on the case where the subgroup is split normal) and Mackey machine for semidirect product of abelian groups, which are needed for the description of unitary dual of the Heisenberg group in \S\ref{secdual}. This material is well known and available in introductory harmonic analysis books, except perhaps that we work over a general local field instead of $\r$. The last subsection is on Fourier-Wigner Transform which is the main tool in our arguments. Section \ref{secfindim} is about finite dimensional representations of second countable locally compact amenable groups and  section \ref{secheisenberg} contains the results about Heisenberg groups over local fields.

% Thus enough to have for each $g_1,g_2\in G$, 
% \begin{equation*}
%     \underset{n\to\infty}{\lim}\frac{1}{a_n}(\underset{g_2^{-1}F_ng_1\Delta F_n}{\int}|\langle\pi(g)v_1,v_2\rangle|^2d\mu)=0.
% \end{equation*}

% In order to create a general framework for deriving `asymptotic Schur orthogonality relations', Kazhdan and Yom Din \cite{kyd} formulated the following definition. 

% \begin{definition}
    
% \end{definition}

% c-tempered representations sastisfy the following asymptotic Schur orthogonality relations.

% \begin{prop}
% \end{prop}

% The following is conjectured in \cite{kyd}.

% \vspace{.2cm}

% \textbf{Conjecture.}

% \vspace{.2cm}

% This conjecture has been verified for SL$(2,\r)$ and PSL$(2,\Omega)$, for 
% On the other hand, any second countable locally compact abelian group, being amenable, admits a F\o lner sequence and all its irreducible unitrary representations are tempered. Moreover, being 1 dimensional, they are also c-tempered. Thus the above conjecture is trivially true for second countable locally compact abelian groups. This leads us to think that the above conjecture is in fact true for all locally compact groups. There are two well known cases where this is known.  These two cases may offer some hint for the general conjecture, but it may be quite hard.

% We would also like to verify conjecture for groups that are neither abelian nor semisimple, but whose unitary dual is well understood. 

\section{Preliminaries on c-temperedness and  braiding operator}\label{prelimctemp}
\subsection{c-temperedness and asymptotic Schur orthogonality}\label{ctemp} 
The following definition is due to Kazhdan and Yom-Din. 
%Kazhdan and Yom Din introduced c-temperedness for irreducible unitary representations of \textit{any} unimodular locally compact group, as a bridge between temperedness and asymptotic Schur orthogonality (in a particular sense). They showed that c-tempered representations are tempered and satisfy asymptotic Schur orthogonality relations and asked for which groups tempered representations are c-tempered \cite[Question 1.17]{kyd}. The notion of c-temperedness is in fact equivalent to the existence of (a version of) asymptotic Schur orthogonality relations, see Proposition \ref{ctempeqaso} below.

%Let us recall the definition of c-temperedness.
\begin{definition}\cite[Definition 2.1]{kyd}
Let $(\pi,V_\pi)$ be an irreducible unitary representation of a locally compact unimodular group $G$. Let $\{F_n\}_n$ be a sequence of measurable pre-compact subsets containing a neighbourhood of $1$. We say
that $\pi$ is \textit{c-tempered with F\o lner sequence $\{F_n\}_n$} if there exists a unit vector $v_0\in V_\pi$ such that

i. For all $v_1,v_2\in V_\pi$
\begin{equation*}
\limsup_{n\to\infty}\frac{\int_{F_n}|\langle\pi(g)v_1,v_2\rangle|^2d\mu(g)}{\int_{F_n}|\langle\pi(g)v_0,v_0\rangle|^2d\mu(g)}<\infty.
\end{equation*}

ii. For all $v_1,v_2\in V_\pi$ and all compact subsets $C\subset G$ we have
\begin{equation*}
\lim_{n\to\infty}\frac{\sup_{g_1,g_2\in C}\int_{g_2^{-1}F_ng_1\Delta F_n}|\langle\pi(g)v_1,v_2\rangle|^2d\mu(g)}{\int_{F_n}|\langle\pi(g)v_0,v_0\rangle|^2d\mu(g)}=0.
\end{equation*}
\end{definition}

% \begin{prop}\label{ctempeqaso}
% Let $G$ be a unimodular locally compact group. An irreducible unitary representation $(\pi,V_\pi)$ is c-tempered with F\o lner sequence $\{F_n\}_n$ if and only if it satisfies the following asymptotic Schur orthogonality relations. For any $v_1,v_2\in V_\pi$
% \begin{equation}\label{asofol}
% \lim_{n\to\infty}\frac{\int_{F_n}\langle\pi(g)v_1,v_2\rangle\overline{\langle\pi(g)w_1,w_2\rangle}d\mu(g)}{\int_{F_n}|\langle\pi(g)v_0,v_0\rangle|^2d\mu(g)}=\langle v_1,w_1\rangle\overline{\langle v_2,w_2\rangle},
% \end{equation}
% where $v_0$ is as in the definition of c-temperedness.
% \end{prop}
% \begin{proof}
% ($\Rightarrow$) This implication is shown by Kazhdan and Yom-Din \cite[Proposition 2.3]{kyd}

% ($\Leftarrow$) Condition i. is immediate by putting $v_1=w_1$ and $v_2=w_2$ in (\ref{asofol}). For condition (ii), note that for fixed $g_1,g_2\in G$,
% \begin{align*}
% &\lim_{n\to\infty}\frac{\int_{g_2^{-1}F_ng_1}|\langle\pi(g)v_1,v_2\rangle|^2d\mu(g)}{\int_{F_n}|\langle\pi(g)v_0,v_0\rangle|^2d\mu(g)}=\lim_{n\to\infty}\frac{\int_{F_n}|\langle\pi(g_2gg_1^{-1})v_1,v_2\rangle|^2d\mu(g)}{\int_{F_n}|\langle\pi(g)v_0,v_0\rangle|^2d\mu(g)}\\
% =~&\lim_{n\to\infty}\frac{\int_{F_n}|\langle\pi(g)\pi(g_1^{-1})v_1,\pi(g_2^{-1})v_2\rangle|^2d\mu(g)}{\int_{F_n}|\langle\pi(g)v_0,v_0\rangle|^2d\mu(g)}=\|\pi(g_1^{-1})v_1\|^2\|\pi(g_2^{-1})v_2\|^2=\|v_1\|^2\|v_2\|^2.
%  \end{align*}
% \end{proof}
\noindent Such representations satisfy asymptotic Schur orthogonality relations \cite[Proposition 2.1]{kyd} which essentially say that a much stronger form of condition i. holds: if $v_1,v_2$ and $w_1,w_2$ are unit vectors then the ratio of the integrals of the norm squares of the respective matrix coefficients over $F_n$ are \textit{asymptotically equal}:
\begin{equation*}
\lim_{n\to\infty}\frac{\int_{F_n}|\langle\pi(g)v_1,v_2\rangle|^2d\mu(g)}{\int_{F_n}|\langle\pi(g)w_1,w_2\rangle|^2d\mu(g)}=1.
\end{equation*}
%Equivalently, if $v_1,v_2,w_1,w_2$ are unit vectors then the sequences $\{\int_{F_n}|\langle\pi(g)v_1,v_2\rangle|^2d\mu(g)\}_n$ and $\{\int_{F_n}|\langle\pi(g)w_1,w_2\rangle|^2d\mu(g)\}_n$ are \textit{asymptotically equal}, as in the limit of their ratio goes to 1.
On the other hand condition i. can be thought of as a weaker form of equivalence defined as follows. 
%This suggests that condition i. in the definition of c-temperedness can be thought of as the requirement that for all non-zero matrix coefficients the integrals of their norm squares over $F_n$ grow ``similarly''. This similarity can be expressed as an equivalence relation on the set of non-negative sequences.
%The numbers $\int_{F_n}|\langle\pi(g)v_0,v_0\rangle|^2d\mu(g)$ act as \textit{scaling factors} depending on the unit vector $v_0$. A different choice of scaling factors corresponding to a different unit vector would still yield (i) and (ii), \cite[Proposition 2.4]{kyd}. But this is not clear a priori, the result is proved using the asymptotic Schur orthogonality for c-tempered representations \cite[Proposition 2.3]{kyd}. We would like a definition that emphasizes the sequence of scaling factors as a data that comes along with the definition of c-temperedness, along with the F\o lner sets data.
%Our first goal is to give an equivalent definition of c-temperedness that suggests a natural way of comparing two c-tempered representations with the same F\o lner sequence. 
%For any $v_1,v_2\in V_\pi$, $\{\int_{F_n}|\langle\pi(g)v_1,v_2\rangle|^2d\mu(g)\}_n$ is a sequence of positive numbers. We will show that condition (i) in the above definition is equivalent to saying that all such sequences are in the same class with respect to a certain equivalence relation.
%We begin by defining 
On the set of non-negative sequences define a  pre-order $\{a\}_n\preceq\{b_n\}_n$ if $\limsup_{n\to\infty}a_n/b_n<\infty$ (or equivalently $\liminf_{n\to\infty}b_n/a_n>0$). This induces an equivalence relation given by $\{a_n\}_n\sim\{b_n\}_n$ if
\begin{equation}\label{seqeqrel}
    0<\liminf_{n\to\infty}\frac{a_n}{b_n}\le \limsup_{n\to\infty}\frac{a_n}{b_n}<\infty.
\end{equation}
In this case, we say $\{a_n\}_n$ and $\{b_n\}_n$ are \textit{asymptotically equivalent}.

\begin{lemma}\label{zerolim}
Let $\{a_n\}_n$ and $\{b_n\}_n$ be two asymptotically equivalent sequences. For any sequence $\{c_n\}_n\subset\r$, if $\lim_{n\to\infty}c_n/a_n=0$ then $\lim_{n\to\infty}c_n/b_n=0$.
\end{lemma}
\begin{proof}
Taking limit $n\to\infty$ in the following inequality implies the statement. 
\begin{equation*}
    \frac{c_n}{a_n}\inf_{m\ge n}\frac{a_m}{b_m}\le\frac{c_n}{b_n}\le\frac{c_n}{a_n}\sup_{m\ge n}\frac{a_m}{b_m}\qedhere
\end{equation*}
\end{proof}
\noindent It is noted in \cite[Proposition 2.4]{kyd} that the choice of the unit vector in the definition of c-temperedness is immaterial. It follows from the above discussion that a reformulation of the definition, without choosing a unit vector $v_0$, is as follows:   
%\begin{prop}\label{ctempalt}
%An irreducible unitary representation $(\pi,V_\pi)$ of a locally compact unimodular group $G$ is c-tempered with F\o lner sequence $\{F_n\}_n$ if and only if the following two conditions are satisfied

(i) The sequences $\{\int_{F_n}|\langle\pi(g)v_1,v_2\rangle|^2d\mu(g)\}_n$, as $v_1,v_2$ vary in $V_\pi\setminus\{0\}$, are all asymptotically equivalent.   
%For all $v_1,v_2\in V_\pi$
%\begin{equation*}
% 0<\liminf_{n\to\infty}\frac{1}{a_n}\int_{F_n}|\langle\pi(g)v_1,v_2\rangle|^2d\mu(g)\le \limsup_{n\to\infty}\frac{1}{a_n}\int_{F_n}|\langle\pi(g)v_1,v_2\rangle|^2d\mu(g)<\infty.
% \end{equation*}

(ii) For all compact $C\subset G$,
\begin{equation*}
\lim_{n\to\infty}\frac{1}{a_n}\sup_{g_1,g_2\in C}\int_{g_2^{-1}F_ng_1\Delta F_n}|\langle\pi(g)v_1,v_2\rangle|^2d\mu(g)=0,
\end{equation*}
for some, and hence all (by Lemma \ref{zerolim}), $\{a_n\}_n$ in the equivalence class given by (i).
\noindent We state the abstract asymptotic Schur orthogonality results of Kazhdan and Yom Din in a form convenient for us.
%As such the following form   are shown   can be   practise, at least in our case, a slightly stronger form of c-temperedness is established in the framework of the following result. The statement and its proof, although not stated in this form, is basically from \cite{kyd}. See \S4.1 of loc. cit. This form has the added advantage of obtaining the precise constant in the statement of asymptotic Schur orthogonality relations. 

\begin{prop}\cite[Propositions 2.3 and 2.5]{kyd}\label{aso}
Let $G$ be a locally compact unimodular group with Haar a measure $\mu$. Let $(\pi,V_\pi)$ be a c-tempered representation of $G$ with F\o lner sequence $\{F_n\}_n$.  Then for all unit vectors $v_1,v_2\in V_\pi$ the sequences $\{\int_{F_n}|\langle\pi(g)v_1,v_2\rangle|^2d\mu(g)\}_n$ are asymptotically equal.  Any sequence in this class of asymptotically equal sequences will be called a \emph{scaling sequence for $\pi$}. For any scaling sequence $\{a_n\}_n$, it follows via the polarization identity, that 
% such that, for all $v_1,v_2\in V_\pi$ and $C\subset G$ compact, the following hold:
% \begin{enumerate}[i.]
%     \item $\underset{n\to\infty}{\lim}\frac{1}{a_n}\int_{F_n}|\langle\pi(g)v_1,v_2\rangle|^2d\mu=c(v_1,v_2)$, where $c(v_1,v_2)$ is positive, and 
%     \item $\underset{n\to\infty}{\lim}\frac{1}{a_n}\sup_{g_1,g_2\in C}\underset{g_2^{-1}F_ng_1\Delta F_n}{\int}|\langle\pi(g)v_1,v_2\rangle|^2d\mu=0.$
% \end{enumerate}
%Then the following asymptotic Schur orthogonality relations hold. F
for $v_1,w_1,v_2,w_2\in V_\pi$,
\begin{equation*}
\lim_{n\to\infty}\frac{1}{a_n}\int_{F_n}\langle\pi(g)v_1,v_2\rangle\overline{\langle\pi(g)w_1,w_2\rangle}d\mu=\langle v_1,w_1\rangle\overline{\langle v_2,w_2\rangle}, 
\end{equation*}
%where $c_\pi>0$ is a constant whose value is $\underset{n\to\infty}{\lim}\frac{1}{a_n}\int_{F_n}|\langle\pi(g)v_0,v_0\rangle|^2d\mu$ for any unit vector $v_0$. 
Suppose $(\rho,V_\rho)$ is another c-tempered representation with the same F\o lner sequence and a (possibly different) scaling sequence $\{b_n\}_n$. If $\rho$ is not equivalent to $\pi$ then for any $v_1,v_2\in V_\pi$ and $w_1,w_2\in V_\rho$,
\begin{equation*}
\lim_{n\to\infty}\frac{1}{\sqrt{a_nb_n}}\int_{F_n}\langle\pi(g)v_1,v_2\rangle\overline{\langle\rho(g)w_1,w_2\rangle}d\mu=0.
\end{equation*}
% Moreover, if $\pi$ satisfies i. and ii., then for any choice of unit vector $v_0$, it is c-tempered with F\o lner sequence $\{F_n\}_n$ 
\end{prop}

\subsection{Convergence to the braiding operator} Let $G$ be a locally compact group and let $(\pi,V_\pi)$ be a unitary representation of $G$. Let $V_\pi\otimes V_\pi$ denote the Hilbert tensor product and let $F$ be the operator on this space that takes $v\otimes w\to w\otimes v$, for all $v,w\in V_\pi$. The operator $F$ is called the \textit{flip} or the \textit{braiding} operator. The following result is a slight modification of the one due to Bendikov, Boyer and Pittet. 
\begin{prop}\label{braid}\cite[Proposition 1.1]{bbp} Let $(\pi,V_\pi)$ be a unitary representation of a locally compact group $G$ and $\{\mu_n\}_n$ be a family of bounded  Borel measures on $G$. Among the following statements (1) and (2) are equivalent and they both imply (3).

(1) There is a uniform bound on the operator norms
\begin{equation*}
\sup_{n\in\n}\|\int_G\pi(g^{-1})\otimes\pi(g)d\mu_n\|<\infty
\end{equation*}
and a dense subset $W\subset V_\pi$, such that for all $v_1,w_1,v_2,w_2\in W$
\begin{equation*}
\lim_{n\to\infty}\int_G\langle\pi(g)v_1,w_1\rangle\overline{\langle\pi(g)v_2,w_2\rangle}d\mu_n=\langle v_1,v_2\rangle\overline{\langle w_1,w_2\rangle}.
\end{equation*}

(2) In the weak operator topology,
\begin{equation*}
\lim_{n\to\infty}\int_G\pi(g^{-1})\otimes\pi(g)d\mu_n=F
\end{equation*}

(3) For all $v_1,v_2,w_1,w_2\in V_\pi$,
\begin{equation*}
\lim_{n\to\infty}\int_G\langle\pi(g)v_1,w_1\rangle\overline{\langle\pi(g)v_2,w_2\rangle}d\mu_n=\langle v_1,v_2\rangle\overline{\langle w_1,w_2\rangle}.
\end{equation*}
\end{prop}
\noindent Item (2) is the \textit{convergence to the braiding operator} which implies item (3), asymptotic Schur orthogonality. The first condition in item (1) aids in proving item (2) if item (3) is already know.
%Convergence to the braiding operator which is likely a stronger statement than asymptotic Schur orthogonality.
The original statement in \cite{bbp} considered averages of the operators $\pi(g)\otimes\pi(g^{-1})$ instead of $\pi(g^{-1})\otimes\pi(g)$. Because of this they had to use the change of variable $g\mapsto g^{-1}$ in the proof of (2)$\Rightarrow$(3), which was facilitated by their assumption that $\mu_n$ are symmetric measures. Since we start with the operator $\pi(g^{-1})\otimes\pi(g)$ we could remove the assumption that $\mu_n$ are symmetric. 

%%%%%%%%%%%%%%%%%%%%%%%%%%%%%%%%%%%%%%%%%%%%%%%%%

\section{Preliminaries on Heisenberg groups}\label{prelimheisenberg}
\subsection{Induced representation.} Let $H$ be a unimodular locally compact group and $N$ be a unimodular closed subgroup. Let $(\sigma,V_\sigma)$ be a unitary representation of $N$. The Hilbert space for the \textit{induced representation}  $\text{Ind}_N^H\sigma$ of $H$ is the space of $L^2$-sections of the bundle $H\times_NV_\sigma\to H/N$, whose total space is the quotient $H\times V_\sigma/\sim$ by the equivalence relation $(hn,v)\sim (h,\sigma(n)v)$, for all $h\in H,n\in N$ and $v\in V_\sigma$ and the projection is induced by projection to first coordinate of $H\times V_\sigma$. The left translation action induces actions of $H$ on $H/N$ and $H\times_N V_\sigma$ which commute with the projection map. Thus $H$ acts on this bundle via bundle isomorphisms. The induced representation is then given by \begin{equation*}
    (\text{Ind}_N^H\sigma(h)f)(xN):=h\cdot f(h^{-1}xN).
\end{equation*}
We will be interested in the case where $H$ is a semidirect product of the form $N\rtimes G$. In this case, the above bundle is isomorphic to the trivial bundle $G\times V_\sigma\to G$. An explicit bundle isomorphism is induced by the map
\begin{align}\label{triv}
    H\times_N V_\sigma&\to G\times V_\sigma\\ 
    [(n,g),v]&\mapsto (g,\sigma(g^{-1}\cdot n)v), \nonumber
\end{align}
whose inverse  is $(g,v)\mapsto [(1,g),v]$. The $H$ action on the LHS of (\ref{triv}) translates to the following $H$ action on the RHS: $(n,g)\cdot (x,v):=(gx,\sigma(gx)^{-1}v)$. Thus we may identify the representation space of $\text{Ind}_N^H\sigma$ with $L^2(G,V_\sigma)$ and the $H$ action given by 
\begin{equation*}
    (\text{Ind}_N^H\sigma(n,g)f)(x):=\sigma(x^{-1}\cdot n)f(g^{-1}x).
\end{equation*}

\subsection{Mackey machine for semidirect product of abelian groups}
Let $G$ be a locally compact abelian group and let $N$ be another such group on which $G$ acts by continuous automorphisms. The $G$ action on $N$ induces a $G$ action on the unitary dual $\widehat{N}$. Assume the existence of a Borel section, that is, there is a Borel set in $\widehat{N}$ which intersects each $G$-orbit in exactly one point. For each orbit choose a representative $\nu$ and consider the stabilizer $G_\nu$ of $\nu$ in $G$. For any irreducible representation $\rho$ of $G_\nu$ consider the representation $\sigma$ on $H_\nu=N\rtimes G_\nu\subset N\rtimes G$ given by $\sigma(n,g)=\nu(n)\rho(g)$. Then $\text{Ind}_{H_\nu}^H\sigma$ is an irreducible representation of $H$. This construction exhausts all irreducible representations of $H$. If we replace $\nu$ by $g\cdot\nu$ and $\rho$ by the representation of the new stabilizer $gH_\nu g^{-1}$ obtained by precomposing $\rho$ by conjugation by $g$, then the resulting induced representation is equivalent to the original.

\subsection{Description of unitary dual of Heisenberg group via Mackey machine}\label{secdual}
Let $K$ be a local field, that is, a non-discrete locally compact field. The  \textit{Heisenberg group} $H_n(K)$ is the following subgroup of $\text{GL}_{n+2}(K)$, written in block matrix form as
\begin{equation}\label{heidef}
 H_n(K):=\{\begin{pmatrix}1 & a^\top & c\\
0 & I_n & b\\
0 & 0 & 1
 \end{pmatrix}~|~a,b\in K^n, c\in K\},
\end{equation}
where $I_n$ denotes the $n\times n$ identity matrix. For brevity we will write $H$ instead of $H_n(K)$.  The matrix in (\ref{heidef}) will be denoted by the tuple $(a,b,c)\in K^n\times K^n\times K$ and we note the multiplication rule $(a_1,b_1,c_1)(a_2,b_2,c_2)=(a_1+a_2,b_1+b_2,c_1+c_2+a_1^\top b_2)$. This implies $H\cong (K^n\times K)\rtimes K^n$, where the $K^n$ action on $K^n\times K$ is given by $a\cdot (b,c)=(b,c+ a^\top b)$. The Haar measure on $H$ coincides with the Haar measure on $K^{2n+1}$, where $H$ is identified to $K^{2n+1}$ as a set. Since $H$ is of the form $N\rtimes G$, where $N:=K^n\times K$ and $G:=K^n$, we may apply the Mackey machine to find every irreducible unitary representation of $H$. For any $m\in\n$, the unitary dual of $K^m$ can be described as follows \cite[Ch.2, \S5, Thm.3]{weil}. %For $x=(x_1,\cdots,x_m),y=(y_1,\cdots,y_m)\in F^n$ define $\langle x,y\rangle:=\sum_{i=1}^mx_iy_i$.
Fix a non-trivial character $\chi$ of $K$. Then any character of $K^m$ is of the form $\chi_y(x):=\chi( y^\top x)$, where $y\in K^m$. This gives an identification of $\widehat{K^m}$ with $K^m$. In particular for $(x,y)\in K^n\times K$, we will denote the corresponding character by $\chi_{(x,y)}$. The action of $G=K^n$ on $\widehat{N}$  is given by 
\begin{equation*}
    (a\cdot\chi_{(x,y)})(b,c)=\chi_{(x,y)}(b,c-a^\top b)=\chi(x^\top b+cy-a^\top b y)=\chi_{(x-ay,y)}(b,c).
\end{equation*}
Thus the orbits are either points on the subspace $y=0$ or hyperplanes of the form $y=t$ where  $t\ne 0$. The union of the line $x=0$ and the subspace $y=0$ is a Borel section. For the singleton orbits we have $G_{(x,0)}=G\cong K^n$ and any irreducible representation is a character $a\mapsto\chi(z^\top a)$ for some $z\in K^n$. Thus the irreducible representations of $H$ obtained from these singleton orbits are one dimensional of the form $\rho_{z,x}$, where $z,x\in K^n$, given by
\begin{equation}\label{1d}
    \rho_{z,x}(a,b,c):=\chi(z^\top a+x^\top b).
\end{equation}

\noindent For the orbit given by the hyperplane $y=t$, we choose the representative $(0,t)$. Then $G_{(0,t)}=\{0\}$ and hence $H_{(0,t)}=N$ and $\sigma=\nu_t$ is the character $(b,c)\mapsto\chi(tc)$. We will denote $\text{Ind}_N^H\nu_t$ by $\pi_t$. Then Mackey theory implies that all infinite dimensional unitary irreducible representations of $H$ are on $L^2(K^n)$ and are of the form
\begin{equation}\label{infirrephei}
    (\pi_t(a,b,c)f)(x)=\chi(t(c-x^\top b))f(x-a).
\end{equation}

\subsection{The Fourier-Wigner Transform}\label{secfw}  The reference for this subsection is \cite[Chapter 1, \S4]{folland2}, where only Heisenberg groups over $\r$ are considered. We will change $\r$ to a local field $K$ and note that all the arguments still go through. What we will call the Fourier-Wigner Transform will be slightly different from Folland's definition. The transform is motivated by consideration of the matrix coefficient of $\pi_t$, as defined in (\ref{infirrephei}).
%, can be written in the following form.
%via a change of variable,
%$x\mapsto x+a/2$, for the second equality, we have
\begin{align*}
\langle\pi_t(a,b,c)f_1,f_2\rangle
=&\int_{K^n}\chi(t(c-x^\top b))f_1(x-a)\overline{f_2(x)}dx\\
=~&\chi(tc)\int_{K^n} \chi(-tx^\top b)f_1(x-a)\overline{f_2(x)}dx  
\end{align*}
In view of the isometric isomorphism $L^2(K^n)\otimes L^2(K^n)\cong L^2(K^{2n})$, which sends $f_1\otimes f_2$ to the function $(x,y)\mapsto f_1(x)f_2(y)$, we consider the transform $V$ on $L^2(K^{2n})$ given by $V(\phi)(a,b):=\int_{K^n}\chi(-x^Tb)\phi(x-a,x)dx$. We will call $V$ the \textit{Fourier-Wigner transform}.
%and is denoted by $V(f_1,f_2)$. 
%We recall in detail the following well known result about this transform \cite[Chapter 1, \S4]{folland2}. 
%The bilinear map $(f_1,f_2)\mapsto V(f_1,\overline{f_2})$ extends to the Hilbert tensor product $L^2(K^n)\otimes L^2(K^n)$.
%, which is the closure of the algebraic tensor product with respect to the inner product given by $\langle f_1\otimes f_2,f_3\otimes f_4\rangle:=\langle f_1,f_3\rangle\langle f_2,f_4\rangle$.
%As Hilbert spaces, $L^2(K^n)\otimes L^2(K^n)$ and $L^2(K^{2n})$ are isometrically isomorphic, via the map that sends $f_1\otimes f_2$ to the function $(x,y)\mapsto f_1(x)f_2(y)$. The extension $\tilde{V}$ of $V$ to $L^2(K^n)\otimes L^2(K^n)\cong L^2(K^{2n})$
We claim that $V$ is a unitary operator. This becomes clear if we see $V$ as the composition of the following two operators, which are clearly unitary. The first one is obtained by precomposing functions in $L^2(K^{2n})$ by the automorphism of $K^{2n}$ given by $(a,b)\mapsto (-a+b,b)$. We claim that this automorphism is measure preserving. The pull back of a given Haar measure of a group $G$ by an automorphism $\phi$ is again a Haar measure and hence is positive scalar multiple of the given Haar measure. This scalar is called the \textit{modulus of the automorphism} and is denoted by $\text{mod}_G(\phi)$. If $T$ is an invertible linear transform of $K^m$, where $K$ is a local field, then $\text{mod}_{K^m}(T)=\text{mod}_K(\det T)$ \cite[Cor.3 of Thm.3, Ch.1]{weil}. This proves the claim and shows that corresponding operator is unitary. The second operator is better seen as one acting on $L^2(K^n)\otimes L^2(K^n)$ which sends $f_1\otimes f_2$ to $f_1\otimes\widehat{f_2}$, where the Fourier transform $\widehat{f_2}$ of $f_2$ is again a function on $K^n$ via the identification of $\widehat{K^n}$ with $K^n$ using $\chi$. We may adjust the Haar measure on $K$ so that the Fourier transform on $L^2(K^n)$ continues to be an isometry. This proves that the second operator is unitary too. Note that the same arguments hold for the following slightly modified operator.
\begin{equation}\label{fw2}
    V'(\phi)(a,b):=\int_{K^n}\chi(-x^Tb)\phi(a-x,x)dx.
\end{equation}

\noindent The unitarity of the Fourier-Wigner transform has the following consequence for the matrix coefficients of $\pi_t$.
\begin{equation}\label{fwimp}
\int_{K^n}\int_{K^n}|\langle\pi_t(a,b,c)f_1,f_2\rangle|^2da~db
%=\frac{1}{\text{mod}_K(t)^n}\int_{K^n}\int_{K^n}|V(f_1\otimes f_2)(a,b)|^2da~db
=\frac{1}{\text{mod}_K(t)^n}\|f_1\|^2\|f_2\|^2.
\end{equation}

%*****************************************
\section{Finite dimensional representations of amenable groups}\label{secfindim}
\noindent By a result of Greenleaf \cite[Theorem 3.1]{greenleaf},
any continuous action of a locally compact amenable group $G$ on a locally compact space $Z$ is amenable. Moreover, if $Z$ admits a non-zero invariant measure class then the corresponding space $L^\infty(Z)$ admits a $G$ invariant mean. Applying this fact to the action of $G\times G$ on $G$ given by $(g,h)\cdot x:=gxh^{-1}$, we note that $G$ admits a bi-invariant mean $m$ on $L^\infty(G)$. The following lemma comes out of the proof of Schur orthogonality relations for compact groups, by replacing compactness with amenability and the Haar measure with a bi-invariant mean. We include the proof for completeness sake.
\begin{lemma}\label{protoaso}
Let $G$ be a locally compact amenable group with a bi-invariant mean $m$. Let $(\pi,V_\pi)$ and $(\rho,V_\rho)$ be  irreducible unitary representations of $G$. Then for all $v_1,v_2\in V_\pi$ and $w_1,w_2\in V_\rho$,
 \begin{equation*}
m(\langle\pi(g)v_1,v_2\rangle\overline{\langle\rho(g)w_1,w_2\rangle})=\begin{cases}
    c_\pi\langle v_1,w_1\rangle\overline{\langle v_2,w_2\rangle} & \text{if } \pi=\rho,\\
    0 & \text{if } \pi\not\cong\rho,
\end{cases}
 \end{equation*}
where  $c_\pi\ge 0$ is a constant that is not equal to $0$ if and only if $\pi$ is finite dimensional, in which case $c_\pi=1/\dim V_\pi$.
\end{lemma}
\begin{proof}
Let $\overline{V_\pi}$ be denote a Hilbert space with the same additive structure as $V_\pi$, but with scalar multiplication and inner product given by $\lambda\cdot v:=\overline{\lambda}v$ and $(v,w):=\langle w,v\rangle$. Consider the quadlinear form 
\begin{equation*}
D: V_\pi\times\overline{V_\pi}\times \overline{V_\rho}\times V_\rho \to\c,~ D(v_1,v_2,w_1,w_2):=m(\langle\pi(g)v_1,v_2\rangle\overline{\langle\rho(g)w_1,w_2\rangle}).
\end{equation*}
For fixed vectors $v_2\in V_\pi$ and $w_2\in V_\rho$ the form $D(\underline{\hspace{.2cm}},v_2,\underline{\hspace{.2cm}},w_2)$ is $G$-invariant and Schur's lemma applies. If $\pi\not\cong\rho$ then for any $v_1\in V_\pi$ and $w_1\in V_\rho$, $D(v_1,v_2,w_1,w_2)=0$, as required. If $\pi=\rho$ then there exists a constant $c_{v_2,w_2}$ such that for any  $v_1,w_1\in V_\pi$, 
\begin{equation}\label{13}
    D(v_1,v_2,w_1,w_2)=c_{v_2,w_2}\langle v_1,w_1\rangle.
\end{equation}
Similarly fixing $v_1,w_1$, we get a constant $d_{v_1,w_1}$ such that for any $v_2,w_2\in V_\pi$,
\begin{equation}\label{24}
D(v_1,v_2,w_1,w_2)=d_{v_1,w_1}\overline{\langle v_2,w_2\rangle}.
\end{equation}
Let $v_0\in V_\pi$ be a unit vector. Putting $v_2=v_0=w_2$ in (\ref{13}) and (\ref{24}) we get $d_{v_1,w_1}=c_{v_0,v_0}\langle v_1,w_1\rangle$. Putting $c_\pi:=c_{v_0,v_0}$ and substituting this in (\ref{24}) we get
\begin{equation}\label{preaso}
    D(v_1,v_2,w_1,w_2)=c_\pi\langle v_1,w_1\rangle\overline{\langle v_2,w_2\rangle}.
\end{equation}
Now we follow the argument in \cite[\S3]{bbp}. Let $e_1,\cdots,e_k$ be an orthonormal family in $V_\pi$. For $0\ne v\in V_\pi$, 
\begin{equation}\label{onexp}
\sum_{i=1}^k|\langle\pi(g)v,e_i\rangle|^2\le\|\pi(g)v\|=\|v\|^2.   
\end{equation}
Applying $m$ on both sides and using (\ref{preaso}) we get,
\begin{equation*}
kc_\pi\|v\|^2=\sum_{i=1}^kc_\pi\|v\|^2=\sum_{i=1}^km(|\langle\pi(g)v,e_i\rangle|^2)\le\|v\|^2.   
\end{equation*}
If $\pi$ is infinite dimensional then we get that $c_\pi\le 1/k$ for any $k\in\n$, hence $c_\pi=0$. If $\pi$ is finite dimensional, then for $k=\dim V_\pi$ the inequality (\ref{onexp}) is equality. Hence we get $c_\pi=1/\dim V_\pi$. 
\end{proof}

\noindent Another result of Greenleaf says that if a continuous action of an amenable locally compact group $G$ on a locally compact space $Z$ admits an invariant measure $\nu$, then it satisfies the F\o lner condition, that is, given $\epsilon>0$ and a compact $K\subset G$ there exists compact $U\subset Z$ such that $\nu(gU\Delta U)/\nu(U)<\epsilon$ for all $g\in K$. Again applying this facts to the action $(g,h)\cdot x:=gxh^{-1}$ of $G\times G$ on $G$, this time with the assumption that $G$ is unimodular second-countable, we get that $G$ admits a two sided F\o lner sequence (as noted in \cite{kyd}).  Henceforth, $G$ is always assumed to be unimodular and any F\o lner sequence for $G$ is assumed to be two sided. To avoid confusion with F\o lner sequences associated to c-tempered representations we will use the term \textit{group} F\o lner sequence to allude to F\o lner sequences in the usual sense for amenable groups. Group F\o lner sequences are related to bi-invariant means in the following way. Let $\{F_n\}_n$ be any group F\o lner sequence of $G$. For each $n\in\n$, $f\mapsto(1/\mu(F_n))\int_{F_n}fd\mu$ is a positive linear functional on $L^\infty(G)$ which takes the constant function $1$ to $1$. By the Banach-Alaoglu Theorem, passing on to a subsequence if necessary, there exists a mean $m$ on $L^\infty(G)$ such that for all $f\in L^\infty(G)$, $\lim_{n\to\infty}(1/\mu(F_n))\int_{F_n}fd\mu=m(f)$. The F\o lner property of $\{F_n\}_n$ implies that $m$ is invariant. Applying this approximation to the statement of Lemma \ref{protoaso} we get the following result.

\begin{theorem}\label{findim}
Let $G$ be a unimodular second countable amenable group with Haar measure $\mu$ and let  $(\pi,V_\pi)$ be a unitary irreducible representation of $G$. There exists a group F\o lner sequence $\{F_n\}_n$ such that $\pi$ is c-tempered with F\o lner sequence $\{F_n\}_n$ and the sequences $\{\int_{F_n}|\langle\pi(g)v,w\rangle|^2d\mu(g)\}_n$, for all $0\ne v,w\in V_\pi$, are asymptotically equivalent to $\{\mu(F_n)\}_n$ if and only if $\pi$  is finite dimensional. Hence, for finite dimensional irreducible unitary  representations $(\pi,V_\pi)$ and $(\rho,V_\rho)$, for all  $v_1,v_2\in V_\pi$ and $w_1,w_2\in V_\rho$,
\begin{equation}\label{fdaso}
\lim_{n\to\infty}\frac{1}{\mu(F_n)}\int_{F_n}\langle\pi(g)v_1,v_2\rangle\overline{\langle\rho(g)w_1,w_2\rangle}d\mu(g)=\begin{cases}
 \frac{1}{\dim V_\pi}\langle v_1,w_1\rangle\overline{\langle v_2,w_2\rangle} & \text{if }\pi=\rho,\\
 0&\text{if }\pi\not\cong\rho.
\end{cases}
\end{equation}
Moreover, in the weak operator topology \begin{equation*}
    \lim_{n\to\infty}\frac{1}{\mu(F_n)}\int_{F_n}\pi(g^{-1})\otimes\pi(g)~d\mu(g)=\frac{1}{\dim V_\pi}F
\end{equation*}
\end{theorem}

\begin{proof} 
%We first note that condition (ii) (of the alternate definition of c-temperedness as given) in Proposition \ref{ctempalt} is automoatically satisfied when 
The if part and (\ref{fdaso}) follows from Lemma \ref{protoaso} and the discussion preceding this theorem.  The proof of the only if part is a slight modification of that of the corresponding part of Lemma 3.1. 
 By Proposition \ref{aso} there exists a sequence $\{a_n\}_n$, which is asymptotically equivalent to $\{\mu(F_n)\}_n$, such that for all $v_1,v_2,w_1,w_2\in V_\pi$,
\begin{equation}\label{asofolfol}
\lim_{n\to\infty}\frac{1}{a_n}\int_{F_n}\langle\pi(g)v_1,v_2\rangle\overline{\langle\pi(g)w_1,w_2\rangle}d\mu(g)=
\langle v_1,w_1\rangle\overline{\langle v_2,w_2\rangle}. 
\end{equation}
Given $0\ne v\in V_\pi$ and a orthonormal vectors $e_1,\cdots,e_k$ in  $V_\pi$, we have
\begin{equation*}
\sum_{i=1}^k|\langle\pi(g)v,e_i\rangle|^2\le\|\pi(g)v\|=\|v\|^2.   
\end{equation*}
Applying $\limsup_{n\to\infty}(1/a_n)\int_{F_n}~~d\mu$ on both sides and using (\ref{asofolfol}) we get,
\begin{equation*}
k\|v\|^2\le(\limsup_{n\to\infty}\frac{\mu(F_n)}{a_n})\|v\|^2.   
\end{equation*}
Thus $k$ is bounded above and hence $\pi$ is finite dimensional.

\vspace{0.2cm}
\noindent For the convergence to braiding operator, by Proposition \ref{braid}, it is enough to show that $\sup_{n\in\n}\|(1/\mu(F_n))\int_{F_n}\pi(g^{-1})\otimes\pi(g)~d\mu(g)\|$ $<\infty$. But this follows by taking the norm inside the integral and noting that $\pi(g^{-1})\otimes\pi(g)$ being unitary has operator norm equal to $1$. 
\end{proof}

\section{Results for Heisenberg groups}\label{secheisenberg}
\noindent All local fields $K$ admit a norm satisfying $|xy|=|x||y|$, for all $x,y\in K$. The ones which satisfy the \textit{ultrametric property}, given by $|x+y|\le\max\{|x|,|y|\}$ for all $x,y\in K$, are called \textit{non-archimedean} and the rest are called \textit{archimedean}. We fix such a norm on $K$ and for any integer $n>0$, we work with the $\ell^\infty$ norm on $K^n$. Note that with the $\ell^\infty$ norm if $x,y\in K^n$ then $|x^\top y|\le n\|x\|\|y\|$, in fact in the non-archimedean case we have $|x^\top y|\le \|x\|\|y\|$.  Moreover if $|\cdot|$ satisfies the ultrametric property then it continues to hold for the $\ell^\infty$ norm on $K^n$. Balls of radius $r$ and centre $0$ in $K^n$ will be denoted by $B(r)$. Recall that the Heisenberg group $H_n(K)$ can be identified as a set with $K^n\times K^n\times K$. In the result below we will describe subsets of $H_n(K)$ in terms of this identification. The Haar measure on $H_n(K)$ will be denoted by $\mu$ and that on $K^l$ by $\mu_{K^l}$, for any $l$.
%We will abuse notation by denoting the Haar measure of any group by $\mu$, it will be clear from the context which group we are referring to. 
\begin{theorem}\label{hei}
Let $\{r_m\}_m$ be any sequence of positive real numbers such that $r_m\to\infty$. 

i. All irreducible unitary representations of $H_n(K)$ are c-tempered with F\o lner sequence $F_m:=B(r_m)\times B(r_m)\times B(r_m^2)$.

ii. Let $\rho_{z,x}$ and $\pi_t$ be the irreducible representations of $H_n(K)$ as defined in (\ref{1d}) and (\ref{infirrephei}). If $(z_1,x_1)\ne (z_2,x_2)$ then 
\begin{equation*}
    \lim_{m\to\infty}\frac{1}{\mu(F_m)}\int_{F_m}\rho_{z_1,x_1}\overline{\rho_{z_2,x_2}}d\mu=0. 
\end{equation*}
For any $f_1,f_2,f_3,f_4\in L^2(K^n)$,
\begin{align*}
    &\lim_{m\to\infty}\frac{1}{\mu_K(B(r_m^2))}\int_{F_m}\langle\pi_{t_1}(a,b,c)f_1,f_2\rangle\overline{\langle\pi_{t_2}(a,b,c)f_3,f_4\rangle}d\mu(a,b,c)\\
    =~&\begin{cases}\frac{1}{\emph{mod}_K(t_1)^n}
    \langle f_1,f_3\rangle\overline{\langle f_2,f_4\rangle,} &\text{if } t_1=t_2\\
     0 & \text{otherwise}. 
    \end{cases}
\end{align*}
For any $f_1,f_2\in L^2(K^n)$ and any $t\in K^\times, (z,x)\in K^n$, we have
\begin{equation*}
    \lim_{m\to\infty}\frac{1}{\mu_{K^n}(B(r_m))\mu_K( B(r_m^2))}\int_{F_m}\langle\pi_t(a,b,c)f_1,f_2\rangle\overline{\rho_{z,x}(a,b,c)}~d\mu(a,b,c)=0.
\end{equation*}

iii. In weak operator topology
\begin{equation*}
\lim_{m\to\infty}\frac{1}{\mu_K(B(r_m^2))}\int_{F_m}\pi_t(g^{-1})\otimes\pi_t(g)~d\mu(g)=\frac{1}{\emph{mod}_K(t)^n}F.
\end{equation*}
\end{theorem}

\begin{proof} \textit{i. and ii.} That the one dimensional representations are c-tempered and satisfy the asymptotic Schur orthogonality relations will follow
%from Theorem \ref{findim}
once we show that $\{F_m\}_m$ is a F\o lner sequence for $H_n(K)$. For $K=\r$ and $n=1$, it is known that $\{F_m\}_m$ is \textit{one sided} F\o lner. See for instance \cite[Example 4]{janzen}. We will show that it is two sided F\o lner in the general case. In the non-archimedean case, we claim that for a fixed compact subset $C$ and any $g_1,g_2\in C$, the sequence $\{g_2^{-1}F_mg_1\Delta F_m\}_m$ is in fact eventually empty. To see this take $(x_1, y_1, z_1), (x_2, y_2, z_2)\in H_n(K)$ such that the norm on each coordinate of both these elements is bounded by $k$, say. Let $(u,v,w)\in F_m$. We have
\begin{equation}\label{mult}
    (x_1,y_1,z_1)(u,v,w)(x_2,y_2,z_2)=(x_1+u+x_2,y_1+v+y_2,z_1+w+z_2+x_1^\top v+x_1^\top y_2+u^\top y_2).
\end{equation}
The ultrametric property implies, that for large enough $m$, the product continues to be in $B(r_m)\times B(r_m)\times B(r_m^2)$. This proves our claim. In the archimedean case we adopt a strategy similar to that in \cite{kyd}. It follows from (\ref{mult}) that for any $a_1,a_2,k>0$ and $(x_1,y_1,z_1), (x_2,y_2,z_2)\in B(k)\times B(k)\times B(k)$, we have
%there exists $k>0$ such that $g_1,g_2\in B_k^3$. Then for any $r_i>0, 1\le i\le 3$,  
\begin{align}
\begin{split}\label{conj}
&(x_1,y_1,z_1)(B(a_1)\times B(a_1)\times B(a_2))(x_2,y_2,z_2)\\
\subset~& B(a_1+2k)\times B(a_1+2k)\times B(a_2+2k(na_1+1)+nk^2).
\end{split}
\end{align}
For any compact $C\subset H_n(K)$ there exists $k$ such that $C\cup C^{-1}\subset B_k\times B_k\times B_k$. Then it follows from (\ref{conj}), with $a_1=r_m$ and $a_2=r_m^2$, that for any $g_1,g_2\in C$  
%To see this put $g_i=(x_i,y_i,z_i), i=1,2$. Then for any $(u,v,w)\in \prod_{i=1}^3B_{r_i}$ we have
%Then (\ref{conj}) follows by using the individual bounds on the norms of the numbers $x_i,y_i,z_i,u,v,w$,  $i=1,2$. In particular putting $r_1=r_2=n$, and $r_3=n^2$ we have
\begin{equation}\label{cup}
    F_m\cup g_2^{-1}F_mg_1\subset B(r_m+2k)\times B(r_m+2k)\times B(r_m^2+2k(nr_m+1)+nk^2)
\end{equation}
Again, for large  $m$, putting $a_1=r_m-2k, a_2=r_m^2-2k(nr_m+1)-nk^2$ we get
\begin{align*}
 &g_2(B(r_m-2k)\times B(r_m-2k)\times B(r_m^2-2k(nr_m+1)-nk^2))g_1^{-1}\\
 \subset~&  B(r_m)\times B(r_m)\times B(r_m^2-4nk^2)\subset F_m.
\end{align*}
Applying $g_2^{-1}$ on the left and $g_1$ on the right, we get
\begin{equation}\label{cap}
    B(r_m-2k)\times B(r_m-2k)\times B(r_m^2-2k(nr_m+1)-nk^2)\subset F_m\cap g_2^{-1}F_mg_1.
\end{equation}
It follows from (\ref{cup}) and (\ref{cap}) that $\sup_{g_1,g_2\in C}\mu(g_2^{-1}F_mg_1\Delta F_m)$ is less than or equal to the difference between the measure of the set on the RHS of (\ref{cup}) and that on the LHS of (\ref{cap}). The only archimedean local fields are $\r$ and $\c$. For $B(r)\subset K^l$, we have $\mu_{K^l}(B(r))=(2r)^l$ when $K=\r$ and $\mu_{K^l}(B(r))=(\pi r^2)^l$ when $K=\c$. Now a simple calculation shows that
\begin{align*}
   & \lim_{n\to\infty}\frac{\mu(B(r_m+2k)\times B(r_m+2k)\times B(r_m^2+2k(nr_m+1)+nk^2)
)}{\mu(F_m)}\\
=~&\lim_{n\to\infty}\frac{\mu(B(r_m-2k)\times B(r_m-2k)\times B(r_m^2-2k(nr_m+1)-nk^2))}{\mu(F_m)}.
\end{align*}
This implies that $\{F_m\}_m$ is indeed a two sided F\o lner sequence in the archimedean case too. 

\vspace{0.2cm}
\noindent For the infinite dimensional representations $\pi_t$, (\ref{fwimp}) implies
\begin{equation*}
    \lim_{m\to\infty}\frac{1}{\mu_K(B(r_m^2))}\int_{F_m}|\langle\pi_t(a,b,c)f_1,f_2\rangle|^2d\mu(a,b,c)=\frac{1}{\text{mod}_K(t)^n}\|f_1\|^2\|f_2\|^2.
    \end{equation*}
Hence for unit vectors $f_1,f_2\in L^2(K^n)$, the sequences $\{\int_{F_m}|\langle\pi_t(a,b,c)f_1,f_2\rangle|^2d\mu(a,b,c)\}_m$ are all asymptotically equal to the sequence $\{\mu_K(B(r_m^2))/\text{mod}_K(t)^n\}_m$. This already implies the first part of the second equation in \textit{ii.} via the polarization identity. The rest of the asymptotic Schur orthogonality statements will follow from Proposition \ref{aso} once we show that $\pi_t$ are c-tempered. Only condition ii. of c-temperedness remains to be shown. We use the same strategy as that for showing that $\{F_m\}_m$ is a F\o lner sequence. For the non-archimedean case since for any compact $C\subset H_n(K)$ and $g_1,g_2\in C$, the set $g_2^{-1}F_mg_1\Delta F_m$ is empty for large enough $m$, therefore the integration of $|\langle\pi_t(a,b,c)f_1,f_2\rangle|^2$ over this set is $0$ for large enough $m$. This proves condition ii. of c-temperedness in the non-archimedean case. In the archimedean case we note that it is enough to show that the limit of the integration of $|\langle\pi_t(a,b,c)f_1,f_2\rangle|^2$ over the set on the RHS of (\ref{cup}) and that over the set on the LHS of (\ref{cap}),  as $m\to\infty$, are the same.  Again using (\ref{fwimp}) we have
\begin{align*}
&\lim_{m\to\infty}\frac{1}{\mu_K(B(r_m^2))}\underset{B(r_m^2+2k(nr_m+1)+nk^2)}{\int}\underset{B(r_m+2k)}{\int}\underset{B(r_m+2k)}{\int}|\langle\pi_t(a,b,c)f_1,f_2\rangle|^2da~db~dc\\
=~&\lim_{m\to\infty}\frac{\mu_K(B(r_m^2+2k(nr_m+1)+nk^2))}{\text{mod}_K(t)^n\mu_K(B(r_m^2))}\|f_1\|^2\|f_2\|^2=\frac{1}{\text{mod}_K(t)^n}\|f_1\|^2\|f_2\|^2.
\end{align*}
Similarly, the limit as $m\to\infty$ of the integration over the set on the LHS of (\ref{cap}) is also $(1/\text{mod}_K(t)^n)\|f_1\|^2\|f_2\|^2$. Thus condition ii. of c-temperedness holds in the archimedean case too.
 %asymptotic Schur orthogonality relations of involving these representations now follow from Proposition \ref{aso}. 

\vspace{0.2cm}
\noindent iii. The statement claims convergence to the flip operator on the space $\mathcal{B}(L^2(K^n)\otimes L^2(K^n))$. But using the 
isomorphism $L^2(K^n)\otimes L^2(K^n)\cong L^2(K^{2n})$ it is enough to prove the statement for $\mathcal{B}(L^2(K^{2n}))$. We will use the same notation to denote the operators on $\mathcal{B}(L^2(K^{2n}))$ corresponding to $\pi_t(g^{-1})\otimes\pi_t(g)$ and $F$. If $g=(a,b,c)$ and $\phi\in L^2(K^{2n})$ then
\begin{align*}
    ((\pi_t(g^{-1})\otimes\pi_t(g))\phi)(x,y)&:= \chi(t(a+x-y)^\top b)\phi(x+a,y-a),\\
    \text{and }F(\phi)(x,y)&:=\phi(y,x),
\end{align*}
for all $x,y\in K^n$. Let $\phi_1,\phi_2\in L^2(K^{2n})$. In the following calculation we use the change of variable $x\mapsto y-x,y\mapsto x, a\mapsto a+x-y, b\mapsto -b/t$ in the third to fourth line.
\begin{align*}
    %\lim_{m\to\infty}\frac{1}{\mu_K(B(r_m^2))}\langle\int_{F_m}\pi_t(g^{-1})\otimes\pi_t(g)~d\mu(g)\phi_1,\phi_2\rangle=
    &\lim_{m\to\infty}\frac{1}{\mu_K(B(r_m^2))}\underset{F_m}{\int}\langle\pi_t(g^{-1})\otimes\pi_t(g)\phi_1,\phi_2\rangle d\mu(g)\\
    =~&\lim_{m\to\infty}\frac{1}{\mu_K(B(r_m^2))}\underset{F_m}{\int}\underset{K^{2n}}{\int}\chi(t(a+x-y)^\top b)\phi_1(x+a,y-a)\overline{\phi_2(x,y)}dx~dy~d\mu(a,b,c)\\
    %=~&\lim_{m\to\infty}\underset{B(r_m)}{\int}\underset{B(r_m)}{\int}\underset{K^{2n}}{\int}\chi(t(a^\top b+x^\top b-y^\top b))\phi_1(x+a,y-a)\overline{\phi_2(x,y)}dx~dy~da~db\\
    =~&\underset{K^{4n}}{\int}\chi(t(a+x-y)^\top b)\phi_1(x+a,y-a)\overline{\phi_2(x,y)}dx~dy~da~db\\
    =~&\frac{1}{\text{mod}_K(t)^n}\underset{K^{4n}}{\int}\chi(-(a-x)^\top b)\phi_1(a,y-a)\overline{\phi_2(y-x,x)}~dx~dy~da~db\\
    =~& \frac{1}{\text{mod}_K(t)^n}\underset{K^{2n}}{\int}(\underset{K^n}{\int}\chi(-a^\top b)F(\phi_1)(y-a,a)da)(\underset{K^n}{\int}\overline{\chi(-x^\top b)\phi_2(y-x,x)dx})dy~db\\
    =~&\frac{1}{\text{mod}_K(t)^n}\underset{K^{2n}}{\int}V'(F(\phi_1))(y,b)\overline{V'(\phi_2)(y,b)}dy ~db \text{ (see (\ref{fw2})) }=\frac{1}{\text{mod}_K(t)^n}\langle F(\phi_1),\phi_2\rangle\qedhere
\end{align*}
\end{proof}

\begin{remark}
    1. As the proof shows, statement iii. in Theorem \ref{hei} continues to hold if we replace $F_m$ by any sequence of exhausting compact pre-compact sets of the form $A^1_m\times A^2_m\times A^3_m$ and $\mu_K(B(r_m^2))$ by $\mu_K(A^3_m)$. The same can be said about asymptotic Schur orthogonality relations involving a single infinite dimensional representation. But orthogonality relations involving distinct representations require c-temperedness \textit{with the same F\o lner sequence}, which necessitates the sequence $\{F_m\}_m$.

\vspace{0.2cm}
\noindent  2. Let $\phi$ be an automorphism of a locally compact group $G$ and $(\pi,V_\pi)$ be an irreducible unitary representation of $G$. Consider the twisted representation  $(\phi\cdot\pi)(g):=\pi(\phi^{-1}(g))$. If convergence to braiding operator holds for $\pi$ with respect to the sequence of measures $\{\mu_n\}_n$, then the same holds for $\phi\cdot\pi$ with respect to the measures $\{\phi_*\mu_n\}_n$. Similarly, if $\pi$ is c-tempered with F\o lner sequence $\{F_n\}_n$ then $\phi\cdot\pi$ is so with F\o lner sequence $\{\phi^{-1}(F_n)\}_n$. These observations are pertinent for the Heisenberg groups over local fields since any infinite dimensional irreducible representations (up to unitary equivalence) can be obtained from another via twisting by an automorphism. Thus, proving convergence to braiding operator or c-temperedness for one will imply the same for all others, except for the modification of the measures or the F\o lner sequence by the twisting automorphism. The latter in particular is undesirable since orthogonality relations between distinct representations require c-temperedness with the same F\o lner sequence.    
    \end{remark}

\section*{Acknowledgement}
We thank Vivek Tewary for making us aware of the Fourier-Wigner transform. The first author thanks Krea University for hospitality during multiple visits.

\end{document}